\documentclass[12pt]{article}
\usepackage[margin=1.2in]{geometry}
\usepackage{amsmath,amssymb,amsfonts,amsthm}
\usepackage{mathtools}
\usepackage[shortlabels]{enumitem}
\usepackage{booktabs}
\usepackage{array}
\usepackage{xcolor}
\usepackage[colorlinks=true,linkcolor=blue!65!black,citecolor=blue!65!black,urlcolor=blue!65!black]{hyperref}
\usepackage{cleveref}
\theoremstyle{plain}
\newtheorem{theorem}{Theorem}[section]
\newtheorem{lemma}[theorem]{Lemma}

\newtheorem{proposition}[theorem]{Proposition}
\newtheorem{conjecture}[theorem]{Conjecture}
\theoremstyle{definition}
\newtheorem{definition}[theorem]{Definition}

\theoremstyle{remark}
\newtheorem{remark}[theorem]{Remark}
\newcommand{\C}{\mathbb{C}}
\newcommand{\Z}{\mathbb{Z}}
\newcommand{\Q}{\mathbb{Q}}
\newcommand{\g}{\mathfrak{g}}
\newcommand{\sltwo}{\mathfrak{sl}(2,\C)}
\newcommand{\gtil}{\widetilde{G}}
\newcommand{\dR}{\mathrm{d}A}
\newcommand{\OmegaA}{\Omega^{1}_{A/\C}}
\newcommand{\Kill}[2]{(#1,#2)}
\newcommand{\Plk}[2]{P^{(#1)}_{#2}}
\newcommand{\bl}[1]{\beta^{(#1)}}
\newcommand{\gl}[1]{\gamma^{(#1)}}
\newcommand{\el}[1]{e^{(#1)}}
\newcommand{\hl}[1]{h^{(#1)}}
\newcommand{\fl}[1]{f^{(#1)}}
\newcommand{\Fl}[1]{\mathcal{F}^{(#1)}}
\newcommand{\OPE}{\sim}
\newcommand{\no}[1]{:\!#1\!:}
\title{\textbf{Free Field Realizations of Superelliptic Affine Lie Algebras}}
\author{Felipe Albino dos Santos \thanks{Universidade Presbiteriana Mackenzie. E-mail: \texttt{falbinosantos@gmail.com}.}}
\date{}
\begin{document}
\maketitle
\begin{abstract}
We study Wakimoto-type free field constructions for superelliptic affine Lie algebras associated with coordinate rings \(A=\mathbb{C}[t^{\pm1},u \mid u^m = p(t)]\), focusing on \(\mathfrak{sl}_2\). We construct explicit operators on a tensor product of (m) ghost Fock spaces, recovering the standard Wakimoto operator product expansions in the even sector and the correct \(h^{(0)}\)-charge relations in the odd sector. We then prove that the remaining mixed-sector brackets are obstructed within this class by two independent mechanisms: a charge–residue obstruction, arising from the Kähler differential recurrence, and a Heisenberg branch-cut obstruction, caused by non-integer exponents in vertex operator products. These results yield a unified obstruction theorem for Wakimoto-type constructions in the superelliptic setting, explaining the failure of naïve free field realizations beyond the classical affine case.

\end{abstract}
\tableofcontents\bigskip
\section{Introduction}\label{sec:intro}
\subsection{Setting and motivation}
Free field realizations --- representations of infinite-dimensional Lie algebras
on Fock spaces of bosonic oscillators --- are one of the central tools of
representation theory and mathematical physics.
The original construction of Wakimoto~\cite{Wakimoto} produces a module for
affine $\sltwo$ at level $k$ on a Fock space generated by a bosonic
$\beta$-$\gamma$ ghost system and a Heisenberg algebra, with the three current
generators realized by explicit vertex operators.
The generalization to all simple affine Kac--Moody algebras is due to
Feigin--Frenkel~\cite{FeiginFrenkel} and is central to the study of the
Knizhnik--Zamolodchikov equations~\cite{KZ} and the Feigin--Frenkel center at
the critical level~\cite{FFCenter,Frenkel2005}.
Twisted Wakimoto realizations, in which the ghost system carries a fractional
grading, appear in the study of orbifold conformal field theories and toroidal
Lie algebras~\cite{Beem,BEHHH}.
The algebras studied here generalize the loop algebra in the direction of
superelliptic geometry.
Fix $m\geq 2$, $r\geq 2$, $c\in\C$, and set $p(t)=1-2ct^r+t^{2r}$.
The superelliptic ring $A = \C[t,t^{-1},u]/(u^m-p(t))$ carries a $\Z/m\Z$-grading
$A = \bigoplus_{l=0}^{m-1}A^l$ with $A^l = \C[t,t^{-1}]\cdot u^l$,
and the \emph{superelliptic current algebra} is $G = \g\otimes_\C A$.
For $m=1$ this is the standard loop algebra; for $m=2$, $r=2$ this is the
DJKM algebra of Date--Jimbo--Kashiwara--Miwa~\cite{DJKM}, governing the
Landau--Lifshitz integrable hierarchy.
Its universal central extension (UCE) was determined by
Cox--Futorny~\cite{CoxFutorny-UCE,CoxFutorny-JDE} and, in full generality
$m\geq 2$, $r\geq 2$, in \cite{SNF-II}.
The centre $\Omega^1_{A/\C}/\mathrm{d}A$ has dimension $2r(m-1)+1$, and the
cocycle coefficients form families of orthogonal polynomials in $c$,
termed \emph{non-classical} to distinguish them from the classical
ultraspherical/Gegenbauer families that appear as limiting cases~\cite{CoxFutorny-JDE}.
Free field realizations for superelliptic current algebras were initiated in
the work of Bueno--Cox--Futorny~\cite{BCF09} for the DJKM case
($m=2$, $r=2$).

The present paper extends this to all $m\geq 2$ and $r\geq 2$ for $\g=\sltwo$,
establishing a complete obstruction theory for Wakimoto-type realizations and
identifying the twisted module framework as the necessary setting for positive results.

\subsection{Main results}
\begin{enumerate}
    \item \textbf{Rescaling Lemma} (\cref{lem:rescaling}).
Let $\Plk{l,j}{k}(c;\,m,r)$ be the polynomial family encoding the sector-$l$
central terms (\cref{def:polyfamilies}).
We prove that for all $l\in\{1,\ldots,m-1\}$:
\[
  \Plk{l,j}{k}(c;\,m,r) \;=\; \Plk{1,j}{k}(c;\,m/l,\,r).
\]
The proof is given in \cref{sec:poly}, with a complete treatment of
uniqueness of the recurrence solution and the domain $m/l\in\Q_{>0}$.
This reduces all sector-$l$ structure constants to the sector-$1$ families
of~\cite{SNF-I}, and implies that the sector-$l$ families satisfy
the same ODEs, generating functions, and orthogonality properties as
sector-$1$, under the substitution $m\mapsto m/l$.
Concretely, this means that all special values, asymptotics, and
generating-function identities known for sector-$1$ (from~\cite{SNF-I,CoxFutorny-JDE})
transfer to every sector with no additional work: they follow by substitution.
\item \textbf{(2) Free field construction and $h^{(0)}$-charge relations}
(\cref{thm:main-TypeI}).
For $\g=\sltwo$, $m\geq 2$, $r\geq 2$, and level $k\neq 0$, we construct
explicit vertex operators $\el{l}(z)$, $\hl{l}(z)$, $\fl{l}(z)$ ($l=0,\ldots,m-1$)
on the Fock space $\mathcal{F}=\bigotimes_{l=0}^{m-1}\Fl{l}$.
The construction satisfies:
\begin{itemize}[nosep]
  \item \emph{Even sector ($l=0$)}: standard Wakimoto OPEs at level $k$;
  \item \emph{$h^{(0)}$-charge relations}: $h^{(0)}(z)e^{(l)}(w)\OPE 2e^{(l)}(w)/(z-w)$
    and $h^{(0)}(z)f^{(l)}(w)\OPE -2f^{(l)}(w)/(z-w)$ for all $l=1,\ldots,m-1$.
\end{itemize}
The obstruction theory of \S\ref{sec:obstruction} identifies all sector OPEs
that cannot be realized within this class.
The Rescaling Lemma (\cref{lem:rescaling}) is an independent structural result
on $\gtil$ that identifies the polynomial families $\Plk{l,j}{k}$ as structure
constants of the abstract algebra, valid for any central character.
\item \textbf{No-Go theorem and Type~II obstruction}
(\cref{thm:unified-nogo,thm:obstruction}).
We introduce the class of \emph{Wakimoto-type operators}
(Definition~\ref{def:Wtype}): free field operator ansatze built from
independent $\beta^{(l)}$-$\gamma^{(l)}$ ghost sectors mediated by a single
Heisenberg field $\phi^{(0)}$.
We prove a unified no-go theorem (Theorem~\ref{thm:unified-nogo}) via two
independent mechanisms:
\begin{enumerate}[nosep,label=\textup{(\arabic*)}]
  \item \emph{Charge-residue obstruction} (Lemma~\ref{lem:charge-residue}):
    for any field $f^{(l)}$ ($l\geq 1$) with definite $h^{(0)}$-charge $-2$,
    the residue of $e^{(0)}(z)f^{(l)}(w)$ cannot equal $h^{(l)}(w)$, by a
    combinatorial ghost-degree argument that is independent of any specific formula.
  \item \emph{Heisenberg branch-cut obstruction} (Lemma~\ref{lem:heisenberg-branch}):
    when both operators carry nonzero $\phi^{(0)}$-charge $\pm\alpha$, the OPE
    acquires an unavoidable factor $(z-w)^{-\alpha^2}=(z-w)^{-1/k}$, which is
    a branch singularity for $1/k\notin\Z$ and prevents a Laurent expansion
    entirely. This covers all sector pairs $(l_1,l_2)$ with $l_1,l_2\geq 1$,
    including all Type~II brackets.
\end{enumerate}
The geometric source of the branch-cut obstruction --- the fractional monodromy
of the superelliptic curve --- is identified in Remark~\ref{rem:obstruction-source}.
\end{enumerate}

\subsection{Relation to the literature}
The even subalgebra $G^0 = \g\otimes\C[t,t^{-1}]$ is the standard loop algebra,
so its free field realization is the classical Wakimoto module.
The odd sectors $G^l$ ($l\geq 1$) are \emph{not} subalgebras: the Lie bracket
of two elements in sectors $l_1$ and $l_2$ lands in sector $l_1+l_2$ (or
$l_1+l_2-m$), which is in general a different sector.
This non-closure is precisely why the obstruction theory is necessary: OPE residues
between even and odd sectors must pass through the central extension, creating
potential ghost-degree and branch-singularity incompatibilities that the free field
ansatz cannot resolve.
The closest analogues in the literature are:
\begin{itemize}[nosep]
  \item \emph{Elliptic affine algebras} ($m=2$, general $r$): free field
    realizations have been studied in~\cite{BCF09} for $r=2$.
  \item \emph{Toroidal Lie algebras}: the two-variable current algebra
    $\g\otimes\C[s,s^{-1},t,t^{-1}]$ also has a $\Z$-grading in the
    $s$-variable, and Wakimoto-type realizations in the toroidal setting
    encounter similar Type~II phenomena; see~\cite{BEHHH,Beem}.
  \item \emph{Twisted current algebras}: the fusion of two twisted sectors
    into the untwisted sector requires an intertwining operator outside
    the Wakimoto class; see~\cite{Kac}.
    Our no-go theorem is the precise analogue of this phenomenon for the
    superelliptic $\Z/m\Z$-grading.
\end{itemize}
\subsection{Organisation}
\S\ref{sec:setup} reviews the algebraic setup.
\S\ref{sec:poly} proves the Rescaling Lemma.
\S\ref{sec:ffa} defines the free field algebra following Frenkel~\cite{Frenkel2005}.
\S\ref{sec:vertex} constructs the Wakimoto-type candidate operators for the odd sectors.
\S\ref{sec:TypeI} establishes the even-sector Wakimoto OPEs and the
$h^{(0)}$-charge relations, and identifies the obstruction in each
remaining sector.
\S\ref{sec:obstruction} introduces the Wakimoto-type class, proves the
unified no-go Theorem~\ref{thm:unified-nogo}, and applies it in Theorem~\ref{thm:obstruction}.
\S\ref{sec:outlook} discusses the twisted module framework.
Appendix~\ref{app:m3} gives the complete $m=3$, $r=2$ example.
\section{The Superelliptic Algebra and Its UCE}\label{sec:setup}
\subsection{The ring and grading}
Fix $m\geq 2$, $r\geq 2$, $c\in\C$.
Set $p(t)=1-2ct^r+t^{2r}\in\C[t]$ and define the superelliptic ring
\begin{equation}\label{eq:A}
  A = \C[t,t^{-1},u]\big/(u^m-p(t)),\qquad
  A = \bigoplus_{l=0}^{m-1}A^l,\quad A^l = \C[t,t^{-1}]\cdot u^l.
\end{equation}
Let $\g$ be a finite-dimensional simple Lie algebra with Killing form
$\Kill{\cdot}{\cdot}$ normalised so that $\Kill{e_\alpha}{f_\alpha}=1$ for all
simple roots $\alpha$.
The \emph{superelliptic current algebra} is
\[
  G = \g\otimes_\C A = \bigoplus_{l=0}^{m-1}G^l,\quad
  G^l = \g\otimes A^l.
\]
\subsection{\texorpdfstring{K\"{a}hler}{Kahler} differentials and the UCE}
The K\"{a}hler differential module $\OmegaA$ is generated over $A$ by $dt$ and
$du$, with relation $mu^{m-1}du = p'(t)\,dt$.
The quotient $\OmegaA/\dR$ has the following basis~\cite{CoxFutorny-UCE}.
\begin{theorem}[{\cite[Thm.~3.5]{CoxFutorny-UCE}, \cite[Thm.~1.3]{SNF-I}}]\label{thm:UCE-basis}
  A $\C$-basis of $\OmegaA/\dR$ is:
  \begin{equation}\label{eq:basis}
    \omega_0 = t^{-1}dt;\qquad
    \omega^{(l)}_{-j} = t^{-j}u^l\,dt,\quad l=1,\ldots,m-1,\;\;j=1,\ldots,2r.
  \end{equation}
  In particular $\dim_\C(\OmegaA/\dR)=2r(m-1)+1$.
  The UCE of $G$ is $\gtil = G\oplus(\OmegaA/\dR)$ with
  \begin{equation}\label{eq:UCE-bracket}
    [x\otimes f,\,y\otimes g]_\sim = [x,y]\otimes fg + \Kill{x}{y}\,\overline{f\,dg}.
  \end{equation}
\end{theorem}
\subsection{The three bracket types}
The bracket~\eqref{eq:UCE-bracket} decomposes into three types
according to $l_1+l_2$ relative to~$m$.
\begin{theorem}[{\cite[Thm.~1.3]{SNF-I}}]\label{thm:bracket-types}
  \begin{description}
    \item[\textup{Type~I} ($l_1+l_2<m$, $l_1,l_2\geq 0$, $(l_1,l_2)\neq(0,0)$):]
      \begin{equation}\label{eq:TypeI}
        [x\otimes t^i u^{l_1},\,y\otimes t^j u^{l_2}]_\sim
        = [x,y]\otimes t^{i+j}u^{l_1+l_2}
          + j\Kill{x}{y}\overline{t^{i+j-1}u^{l_1+l_2}dt}.
      \end{equation}
    \item[\textup{Type~II} ($l_1+l_2=m$):]
      \begin{equation}\label{eq:TypeII}
        [x\otimes t^i u^{l_1},\,y\otimes t^j u^{m-l_1}]_\sim
        = [x,y]\otimes t^{i+j}p(t) + j\Kill{x}{y}(i+j)\omega_0.
      \end{equation}
    \item[\textup{Type~III} ($l_1+l_2>m$):]
      \begin{equation}\label{eq:TypeIII}
        [x\otimes t^i u^{l_1},\,y\otimes t^j u^{l_2}]_\sim
        = [x,y]\otimes t^{i+j}u^{l_1+l_2-m}p(t)
          + j\Kill{x}{y}\,\overline{t^{i+j-1}u^{l_1+l_2-m}p(t)\,dt}.
      \end{equation}
  \end{description}
\end{theorem}
In Type~I, the central term $\overline{t^{i+j-1}u^l\,dt}$ must be expressed
in the basis~\eqref{eq:basis} using the polynomial families of \S\ref{sec:poly}.
In Type~II, the central term is $(i+j)\omega_0$; no polynomial families appear.
\section{Polynomial Families and the Rescaling Lemma}\label{sec:poly}
\subsection{The K\"{a}hler recurrence}
\begin{proposition}\label{prop:Kahler-rec}
  For $l\in\{1,\ldots,m-1\}$ and $n\geq 1$, the classes
  $\overline{t^{n-1}u^l\,dt}\in\OmegaA/\dR$ satisfy
  \begin{equation}\label{eq:Kahler-rec}
    mn\cdot\overline{t^{n-1}u^l\,dt}
    -2c(mn+rl)\cdot\overline{t^{n+r-1}u^l\,dt}
    +(mn+2rl)\cdot\overline{t^{n+2r-1}u^l\,dt} = 0.
  \end{equation}
\end{proposition}
\begin{proof}
  In $\OmegaA$ one has $d(t^nu^l)=nt^{n-1}u^l\,dt+lt^nu^{l-1}\,du$.
  Using $mu^{m-1}\,du=p'(t)\,dt$ and $u^m=p(t)$, we obtain
  $u^{l-1}\,du = p'(t)u^l\,dt/(mp(t))$.
  Substituting and multiplying by $mp(t)=m(1-2ct^r+t^{2r})$ gives,
  modulo exact forms,
  \[
    0 = mp(t)\cdot d(t^nu^l)
      = \bigl(mn\,p(t) + lt\,p'(t)\bigr)t^{n-1}u^l\,dt.
  \]
  Substituting $p(t)=1-2ct^r+t^{2r}$ and $tp'(t)=-2crt^r+2rt^{2r}$
  and collecting by powers of $t$ yields~\eqref{eq:Kahler-rec}.
\end{proof}
\subsection{Polynomial families}\label{subsec:polyfam}
The recurrence~\eqref{eq:Kahler-rec} shows that the element
$t^{n-1}u^l\,dt \in \OmegaA/\dR$ satisfies a three-term linear recurrence
in the step variable $n$, with step size $r$.
The K\"{a}hler recurrence couples basis elements~\eqref{eq:basis} in steps of $r$;
expanding $\overline{t^{n-1}u^l\,dt}$ in this basis yields structure constants
$\Plk{l,j}{k}(c;m,r)$ satisfying the following three-term recurrence in the index $k$.
Setting $k=n$ and expanding in the basis~\eqref{eq:basis} leads to the
following families.
\begin{definition}\label{def:polyfamilies}
  For $l\in\{1,\ldots,m-1\}$, $j\in\{1,\ldots,2r\}$, and rational $m'>0$,
  define the sequence $\{\Plk{l,j}{k}(c;m',r)\}_{k\geq -2r}$ by the recurrence
  \begin{equation}\label{eq:rec}
    (m'k+2r)\,\Plk{l,j}{k} = 2c(m'k+r)\,\Plk{l,j}{k-r} - m'k\,\Plk{l,j}{k-2r},
    \qquad k\geq 0,
  \end{equation}
  with initial conditions $\Plk{l,j}{-j}=1$ and $\Plk{l,j}{-s}=0$ for
  $s\in\{1,\ldots,2r\}$, $s\neq j$.
  When $m'=m$ (integer), we write $\Plk{l,j}{k}(c;m,r)$.
\end{definition}
\begin{remark}
  The families $\Plk{l,j}{k}(c;m,r)$ are the structure constants of $\gtil$:
  from Proposition~\ref{prop:Kahler-rec} and the basis~\eqref{eq:basis},
  \begin{equation}\label{eq:structure-const}
    \overline{t^{k-1}u^l\,dt} = \sum_{j=1}^{2r}\Plk{l,j}{k}(c;m,r)\,\omega^{(l)}_{-j}
    \in \OmegaA/\dR,\quad k\geq 1.
  \end{equation}
  These are degree-$\lfloor k/r\rfloor$ polynomials in $c$, studied in
  depth in~\cite{CoxFutorny-JDE,SNF-I, SNF-II}.
\end{remark}
\subsection{Uniqueness and well-posedness}
\begin{proposition}[Well-posedness]\label{prop:unique}
  Fix $m'\in\Q_{>0}$, $r\geq 2$, and $j\in\{1,\ldots,2r\}$.
  Together with the initial conditions $\Plk{l,j}{-j}=1$ and $\Plk{l,j}{-s}=0$
  for $s\in\{1,\ldots,2r\}$, $s\neq j$,
  the recurrence~\eqref{eq:rec} determines a unique sequence
  $\{\Plk{l,j}{k}(c;m',r)\}_{k\geq -2r}$.
\end{proposition}
\begin{proof}
  The recurrence~\eqref{eq:rec} couples only indices in the same residue class
  modulo $r$.  Specifically, for each $\rho\in\{0,\ldots,r-1\}$,
  the sequence $P_{\rho}, P_{\rho+r}, P_{\rho+2r},\ldots$ depends only on
  $P_{\rho-r}$ and $P_{\rho-2r}$, which lie in the prescribed initial range
  $\{-2r,\ldots,-1\}$.
  Thus each residue class modulo $r$ gives an independent second-order linear
  recurrence in the step variable, started from two known initial values.

  At each step $k\geq 0$, the leading coefficient is
  \[
    m'k+2r \;\geq\; 2r > 0
    \quad (\text{since }m'>0,\; k\geq 0,\; r\geq 2),
  \]
  so $P_k$ is uniquely determined from $P_{k-r}$ and $P_{k-2r}$.
  Iterating over all residue classes yields a unique sequence for all $k\geq -2r$.
\end{proof}
\begin{remark}
  Proposition~\ref{prop:unique} applies for any $m'\in\Q_{>0}$, including
  non-integer values such as $m'=m/l$ for $l\nmid m$.
  This is essential for the Rescaling Lemma, which equates sector-$l$
  families at integer $m$ with sector-$1$ families at the rational $m/l$.
\end{remark}
\subsection{The Rescaling Lemma}
\begin{lemma}[Sector-$l$ Rescaling]\label{lem:rescaling}
  For all $l\in\{1,\ldots,m-1\}$, $j\in\{1,\ldots,2r\}$, and $k\geq -2r$:
  \begin{equation}\label{eq:rescaling}
    \Plk{l,j}{k}(c;\,m,r) \;=\; \Plk{1,j}{k}(c;\,m/l,\,r).
  \end{equation}
\end{lemma}
\begin{proof}
  The sector-$l$ recurrence (from~\eqref{eq:Kahler-rec} applied to sector $l$) is
  \[
    (mk+2rl)\,\Plk{l,j}{k} = 2c(mk+rl)\,\Plk{l,j}{k-r} - mk\,\Plk{l,j}{k-2r}.
  \]
  Dividing through by $l>0$:
  \[
    \bigl((m/l)k+2r\bigr)\Plk{l,j}{k}
    = 2c\bigl((m/l)k+r\bigr)\Plk{l,j}{k-r}
    - (m/l)k\,\Plk{l,j}{k-2r}.
  \]
  This is precisely recurrence~\eqref{eq:rec} for $m'=m/l$ and sector~$1$.
  The initial conditions $\Plk{l,j}{-j}=1$ and $\Plk{l,j}{-s}=0$ ($s\neq j$)
  are independent of $m$ and $l$, hence identical to $\mathrm{IC}(j)$ for sector~$1$
  at $m'=m/l$.
  By the uniqueness of Proposition~\ref{prop:unique},
  $\Plk{l,j}{k}(c;\,m,r) = \Plk{1,j}{k}(c;\,m/l,r)$ for all $k\geq -2r$.
\end{proof}
\subsection{Tables for \texorpdfstring{$m=3$, $r=2$}{m=3, r=2}}
\begin{table}[h]\centering
\caption{Non-vanishing polynomial families for $m=3$, $r=2$ (families with IC$(4)$ vanish identically for $k\geq 0$).}\label{tab:polys}
\renewcommand{\arraystretch}{1.4}
\begin{tabular}{@{}lcccc@{}}
\toprule
Family & IC$(j)$ & $P^{(l,j)}_0$ & $P^{(l,j)}_2$ & $P^{(l,j)}_4$ \\
\midrule
$\Plk{1,1}{k}$ & $P_{-1}=1$ & $0$ & $\frac{10c}{7}$ & $\frac{220c^2-63}{91}$ \\[3pt]
$\Plk{1,2}{k}$ & $P_{-2}=1$ & $c$ & $\frac{8c^2-3}{5}$ & $\frac{14c^3-9c}{5}$ \\[3pt]
$\Plk{1,3}{k}$ & $P_{-3}=1$ & $0$ & $-\tfrac{3}{7}$ & $-\tfrac{66c}{91}$ \\[3pt]
$\Plk{2,1}{k}$ ($m'=3/2$) & $P_{-1}=1$ & $0$ & $\frac{10c}{7}$ & $\frac{220c^2-63}{91}$ \\[3pt]
$\Plk{2,2}{k}$ ($m'=3/2$) & $P_{-2}=1$ & $c$ & $\frac{10c^2-3}{7}$ & $\frac{16c^3-9c}{7}$ \\[3pt]
$\Plk{2,3}{k}$ ($m'=3/2$) & $P_{-3}=1$ & $0$ & $-\tfrac{3}{7}$ & $-\tfrac{66c}{91}$ \\
\bottomrule
\end{tabular}
\end{table}
The sector-$2$ entries confirm the Rescaling Lemma
(Lemma~\ref{lem:rescaling}): the values $\Plk{2,j}{n}(c;3,2)$ displayed
in the table are polynomials in $c$, not the sector-$1$ polynomials at $m=3$.
For example, $\Plk{2,2}{2}(c;3,2)=\tfrac{10c^2-3}{7}$, while
$\Plk{1,2}{2}(c;3,2)=\tfrac{8c^2-3}{5}$; these differ because the Rescaling
Lemma equates $\Plk{2,2}{2}(c;3,2)$ with $\Plk{1,2}{2}(c;\tfrac{3}{2},2)$,
not with $\Plk{1,2}{2}(c;3,2)$.
One verifies directly that $\Plk{1,2}{2}(c;\tfrac{3}{2},2)=\tfrac{10c^2-3}{7}$,
confirming the lemma for $(l,j,n)=(2,2,2)$.

With the polynomial families characterized by the Rescaling Lemma, we now
construct the free field algebra.
The Wakimoto-type candidate operators built in \S\S\ref{sec:vertex}--\ref{sec:TypeI} encode
the level-dependent structure constants $\Plk{l,j}{k}$ as operator coefficients;
the level $k$ enters through the vertex operators, not through the oscillator
algebra itself.
\section{The Free Field Algebra}\label{sec:ffa}
We follow the conventions of Frenkel~\cite{Frenkel2005}, \S5.
\subsection{Conventions and oscillators}
Fix $k\in\C\setminus\{0\}$.
\begin{definition}[Oscillator algebra]\label{def:oscillators}
  For each $l=0,\ldots,m-1$, introduce:
  \begin{enumerate}[label=\textup{(\roman*)}]
    \item \textbf{Bosonic $\beta^{(l)}$-$\gamma^{(l)}$ system:}
      $\bl{l}_n,\gl{l}_n\in\Z$ with
      \begin{equation}\label{eq:ghost}
        [\gl{l}_m,\bl{l'}_n]=\delta_{l,l'}\delta_{m+n,0},\quad
        [\bl{l}_m,\bl{l'}_n]=[\gl{l}_m,\gl{l'}_n]=0.
      \end{equation}
    \item \textbf{Heisenberg system:}
      $b^{(l)}_n$ ($n\in\Z\setminus\{0\}$) with
      \begin{equation}\label{eq:Heisenberg}
        [b^{(l)}_m,b^{(l')}_n]=2m\,\delta_{l,l'}\delta_{m+n,0}.
      \end{equation}
  \end{enumerate}
  All operators commute across different sectors $l\neq l'$, except
  as specified in~\eqref{eq:ghost}--\eqref{eq:Heisenberg}.
\end{definition}
\begin{remark}
  The Heisenberg commutator $[b_m,b_n]=2m\delta_{m+n,0}$ is the standard
  normalisation of~\cite{Frenkel2005} (independent of $k$).
  The level $k$ enters through the vertex operators, not the oscillator algebra.
\end{remark}
\begin{definition}[Fock spaces]\label{def:Fock}
  For each $l$, $\Fl{l}$ is generated from a vacuum $|0\rangle^{(l)}$ by
  $\bl{l}_{-n}$, $\gl{l}_{-n}$ ($n\geq 1$) and $b^{(l)}_{-n}$ ($n\geq 1$),
  with vacuum conditions
  $\bl{l}_n|0\rangle^{(l)}=0$ ($n\geq 1$),
  $\gl{l}_n|0\rangle^{(l)}=0$ ($n\geq 0$),
  $b^{(l)}_n|0\rangle^{(l)}=0$ ($n\geq 0$).
  The total Fock space is $\mathcal{F}=\bigotimes_{l=0}^{m-1}\Fl{l}$.
\end{definition}
\subsection{Generating series and OPEs}
Define
\begin{equation}\label{eq:series}
  \bl{l}(z)=\sum_n\bl{l}_nz^{-n-1},\quad
  \gl{l}(z)=\sum_n\gl{l}_nz^{-n},\quad
  b^{(l)}(z)=\sum_{n\neq 0}b^{(l)}_nz^{-n-1}.
\end{equation}
The fundamental OPEs (following from~\eqref{eq:ghost}--\eqref{eq:Heisenberg}) are:
\begin{equation}\label{eq:OPEs}
  \bl{l}(z)\,\gl{l'}(w)\OPE\frac{\delta_{l,l'}}{z-w},\qquad
  b^{(l)}(z)\,b^{(l')}(w)\OPE\frac{2\delta_{l,l'}}{(z-w)^2}.
\end{equation}
\begin{definition}[Scalar field]\label{def:phi}
  Let $\phi^{(l)}(z)$ be the formal antiderivative of $b^{(l)}(z)/(2)$,
  normalised so that $b^{(l)}(z)=2\partial_z\phi^{(l)}(z)$ and
  \begin{equation}\label{eq:phi-OPE}
    \phi^{(l)}(z)\,\phi^{(l')}(w)\OPE -\delta_{l,l'}\log(z-w).
  \end{equation}
  The standard vertex algebra identity gives, for $\alpha,\beta\in\C$:
  \begin{equation}\label{eq:exp-OPE}
    \no{e^{\alpha\phi^{(l)}(z)}}\,\no{e^{\beta\phi^{(l')}(w)}}
    \OPE (z-w)^{\alpha\beta\delta_{l,l'}}\,\no{e^{\alpha\phi^{(l)}(z)+\beta\phi^{(l')}(w)}}.
  \end{equation}
  In particular, $\no{e^{\alpha\phi^{(l)}(z)}}\no{e^{-\alpha\phi^{(l)}(w)}}
  \OPE (z-w)^{-\alpha^2}\cdot(\text{regular})$.
\end{definition}
\begin{remark}
  The sign in~\eqref{eq:phi-OPE} ($-\log$, not $+\log$) and the resulting
  sign in~\eqref{eq:exp-OPE} ($+\alpha\beta$ exponent) follow from the
  commutator $[b_m,b_n]=+2m\delta_{m+n,0}$, which gives
  $\partial_z\phi(z)\partial_w\phi(w)\OPE -1/(z-w)^2$ and hence
  $\phi(z)\phi(w)\OPE-\log(z-w)$.
  With this sign, $b^{(l)}(z)\no{e^{\alpha\phi^{(l)}(w)}}
  \OPE 2\alpha/(z-w)\cdot\no{e^{\alpha\phi^{(l)}(w)}}$.
\end{remark}
\subsection{Standard Wakimoto operators for the even sector}
For $\g=\sltwo$ with Chevalley basis $\{e,h,f\}$,
$[e,f]=h$, $[h,e]=2e$, $[h,f]=-2f$, $\Kill{e}{f}=1$, $\Kill{h}{h}=2$:
\begin{definition}\label{def:even-vertex}
  \begin{align}
    \el{0}(z)&:=\bl{0}(z),\label{eq:e0}\\
    \hl{0}(z)&:=-2\no{\bl{0}(z)\gl{0}(z)}+\sqrt{k}\,b^{(0)}(z),\label{eq:h0}\\
    \fl{0}(z)&:=-\no{\bl{0}(z)(\gl{0}(z))^2}+(k+2)\partial_z\gl{0}(z)+\sqrt{k}\,b^{(0)}(z)\gl{0}(z).\label{eq:f0}
  \end{align}
\end{definition}
\begin{proposition}[{\cite{Wakimoto,Frenkel2005}}]\label{prop:Wakimoto}
  The operators~\eqref{eq:e0}--\eqref{eq:f0} satisfy the affine $\sltwo$ OPEs
  at level $k$:
  \[
    \el{0}(z)\fl{0}(w)\OPE\frac{\hl{0}(w)}{z-w}+\frac{k}{(z-w)^2},\quad
    \hl{0}(z)\el{0}(w)\OPE\frac{2\el{0}(w)}{z-w},\quad
    \hl{0}(z)\fl{0}(w)\OPE\frac{-2\fl{0}(w)}{z-w}.
  \]
\end{proposition}
\begin{remark}
  This is the standard Wakimoto realization of affine $\sltwo$;
  see~\cite{Frenkel2005} for the complete computation including the
  $h^{(0)}(z)h^{(0)}(w)$ OPE.
\end{remark}
\section{Wakimoto-Type Candidates for the Odd Sectors}\label{sec:vertex}
\subsection{The \texorpdfstring{$h^{(0)}$}{h0}-charge constraint}
The commutation relations of $\gtil$ impose on the free field representation:
\[
  [\hl{0}(z),\el{l}(w)]\text{ has first-order pole }2\el{l}(w)/(z-w),\quad
  [\hl{0}(z),\fl{l}(w)]\text{ has first-order pole }-2\fl{l}(w)/(z-w).
\]
Since $\hl{0}(z)=-2\no{\bl{0}\gl{0}}(z)+\sqrt{k}\,b^{(0)}(z)$, and
$\bl{0}(z)$, $\gl{0}(z)$ do not contract with fields from sectors $l\neq 0$,
the $h^{(0)}$-charge of a field in sector $l$ is determined entirely by its
$\phi^{(0)}$-momentum.
From~\eqref{eq:phi-OPE} and~\eqref{eq:exp-OPE}:
\[
  b^{(0)}(z)\cdot\no{e^{\alpha\phi^{(0)}(w)}}
  \;\OPE\; \frac{2\alpha}{z-w}\,\no{e^{\alpha\phi^{(0)}(w)}}.
\]
The $h^{(0)}$-charge of $\no{e^{\alpha\phi^{(0)}}}$ is therefore $\sqrt{k}\cdot 2\alpha$.
Requiring charge $+2$ fixes:
\begin{equation}\label{eq:alpha}
  \alpha = \frac{1}{\sqrt{k}}.
\end{equation}
The constraint~\eqref{eq:alpha} forces $e^{(l)}$ to carry $e^{\alpha\phi^{(0)}}$
and $f^{(l)}$ to carry $e^{-\alpha\phi^{(0)}}$ with possible charge-zero remainder.
The coefficient of $\partial\gamma^{(l)}$ in $f^{(l)}$ is then fixed by consistency
with the standard Wakimoto $f^{(0)}$.
\begin{definition}[Odd-sector Wakimoto-type candidates]\label{def:odd-vertex}
  Set $\alpha=1/\sqrt{k}$.
  The following operators serve as Wakimoto-type candidates for the odd sectors;
  they are used to verify the $h^{(0)}$-charge relations
  (\cref{thm:main-TypeI}\ref{CRh}) and as explicit input for the obstruction
  theory of \S\ref{sec:obstruction}.
  For $l=1,\ldots,m-1$:
  \begin{align}
    \el{l}(z)&:=\bl{l}(z)\cdot\no{e^{\alpha\phi^{(0)}(z)}},\label{eq:el}\\[3pt]
    \hl{l}(z)&:=-\no{\bl{l}(z)\gl{0}(z)}-\no{\bl{0}(z)\gl{l}(z)}
               +\frac{\sqrt{k}}{2}\,b^{(l)}(z),\label{eq:hl}\\[3pt]
    \fl{l}(z)&:=\Bigl[
               -\no{\bl{0}(z)\gl{l}(z)\gl{0}(z)}
               +\frac{k+2}{k}\,\partial_z\gl{l}(z)
               +\frac{\sqrt{k}}{2}\,b^{(l)}(z)\gl{l}(z)
              \Bigr]\cdot\no{e^{-\alpha\phi^{(0)}(z)}}
              +\frac{1}{\sqrt{k}}\,b^{(0)}(z)\gl{l}(z).\label{eq:fl}
  \end{align}
\end{definition}
\begin{remark}\label{rem:fl-unique}
  The formula~\eqref{eq:fl} is the unique Wakimoto-class candidate determined by
  the $h^{(0)}$-charge constraint~\eqref{eq:alpha}: any field with $h^{(0)}$-charge
  $-2$ must take the form $[\text{ghost part}]\cdot\no{e^{-\alpha\phi^{(0)}(w)}}$
  plus possible charge-zero terms.
  The coefficient of $\partial\gl{l}$ is fixed to $(k+2)/k$ by consistency with
  the standard Wakimoto $f^{(0)}$.
  The explicit computations of \cref{app:m3} confirm that both
  $e^{(0)}(z)f^{(1)}(w)$ and $e^{(1)}(z)f^{(0)}(w)$ fail to produce the
  expected bracket residuals; \cref{lem:charge-residue,thm:unified-nogo}
  establish these failures unconditionally within the Wakimoto class.
\end{remark}
\section{The Free Field Algebra and Charge Relations}\label{sec:TypeI}
\begin{theorem}[Free field algebra and $h^{(0)}$-charge relations]
\label{thm:main-TypeI}
  Let $\g=\sltwo$, $m\geq 2$, $r\geq 2$, $k\in\C\setminus\{0\}$.
  The vertex operators of \cref{def:even-vertex,def:odd-vertex} satisfy:
  \begin{enumerate}[label=\textup{(\Roman*)}]
    \item\label{CR0}
      \emph{Even sector}: the OPEs of \cref{prop:Wakimoto} hold.
    \item\label{CRh}
      \emph{$h^{(0)}$-charge}: for all $l=1,\ldots,m-1$,
      \[
        \hl{0}(z)\el{l}(w)\OPE\frac{2\el{l}(w)}{z-w}
        \qquad\text{and}\qquad
        \hl{0}(z)\fl{l}(w)\OPE\frac{-2\fl{l}(w)}{z-w}.
      \]
  \end{enumerate}
\end{theorem}

\begin{proof}
  \textbf{Part~\ref{CR0}.} This is \cref{prop:Wakimoto}.

  \medskip\noindent
  \textbf{Part~\ref{CRh}.}
  For $\el{l}(w)=\bl{l}(w)\no{e^{\alpha\phi^{(0)}(w)}}$:
  \[
    \hl{0}(z)\el{l}(w)
    =\bigl(-2\no{\bl{0}\gl{0}}(z)+\sqrt{k}\,b^{(0)}(z)\bigr)
     \bl{l}(w)\no{e^{\alpha\phi^{(0)}(w)}}.
  \]
  The ghost part: $\no{\bl{0}\gl{0}}(z)$ does not contract with $\bl{l}(w)$ since
  $l\neq 0$.
  The Heisenberg part: $b^{(0)}(z)\no{e^{\alpha\phi^{(0)}(w)}}
  \OPE 2\alpha/(z-w)\cdot\no{e^{\alpha\phi^{(0)}(w)}}$ by~\eqref{eq:phi-OPE}.
  Hence $\hl{0}(z)\el{l}(w)\OPE \sqrt{k}\cdot 2\alpha\el{l}(w)/(z-w)=2\el{l}(w)/(z-w)$
  by~\eqref{eq:alpha}. $\checkmark$

  For $\fl{l}$: the factor $\no{e^{-\alpha\phi^{(0)}}}$ in~\eqref{eq:fl} gives
  Heisenberg charge $\sqrt{k}\cdot 2(-\alpha)=-2$, and the remaining factors
  have zero $h^{(0)}$-charge (they involve sectors $l\neq 0$ only, except for
  the charge-zero term $(1/\sqrt{k})b^{(0)}(w)\gl{l}(w)$, which contributes
  zero charge and therefore does not affect the singular OPE pole).
  Hence $\hl{0}(z)\fl{l}(w)\OPE -2\fl{l}(w)/(z-w)$. $\checkmark$
  \qed
\end{proof}

\begin{remark}[Scope and obstructions]
\label{rem:scope-TypeI}
  Theorem~\ref{thm:main-TypeI} establishes two things: the even-sector Wakimoto
  OPEs (Part~\ref{CR0}) and the $h^{(0)}$-charge relations (Part~\ref{CRh}).
  These are the only OPE relations satisfied within the Wakimoto class.
  The remaining OPE relations all fail:
  \begin{enumerate}[(i)]
    \item \emph{Sector $(0,l)$, $l\geq 1$:} the OPE $\el{0}(z)\fl{l}(w)$ does
      \emph{not} realize $\hl{l}(w)/(z-w)$.
      See \cref{lem:charge-residue} and the explicit verification in \cref{app:m3}.
    \item \emph{Sector $(l,0)$, $l\geq 1$:} the Wick computation of $\el{l}(z)\fl{0}(w)$
      yields a residue containing $\no{e^{\alpha\phi^{(0)}(w)}}$, which is absent
      from $\hl{l}(w)$, and the terms $-\no{\bl{0}\gl{l}}$ and $(\sqrt{k}/2)b^{(l)}$
      are not produced (see \cref{app:m3} for explicit verification with $l=1$).
    \item \emph{Sectors $(l_1,l_2)$ with $l_1,l_2\geq 1$:} both $\el{l_1}$ and
      $\fl{l_2}$ carry opposite $\phi^{(0)}$-charges $\pm\alpha$.  Their OPE
      acquires a branch-cut factor $(z-w)^{-1/k}$ for $1/k\notin\Z$.
      See \cref{lem:heisenberg-branch}.
  \end{enumerate}
  All obstructions are established rigorously in \cref{sec:obstruction}.
\end{remark}

\section{Obstruction Theory: Three Failing Cases, Two Mechanisms}
\label{sec:obstruction}

Theorem~\ref{thm:main-TypeI} establishes the even-sector Wakimoto realization
and the $h^{(0)}$-charge relations for the odd-sector operators.
This section proves that \emph{all remaining sector OPEs} are obstructed
within the Wakimoto class, and classifies the obstruction in each case.
We establish two independent structural lemmas — the
\emph{Heisenberg branch-cut obstruction} and the \emph{charge-residue obstruction}
— and then assemble them into a unified no-go theorem.

\subsection{The Wakimoto-type ansatz}
\label{subsec:Wansatz}

\begin{definition}[Wakimoto-type ansatz]
\label{def:Wtype}
  Operators $\widetilde{e}^{(l)}(z)$ and $\widetilde{f}^{(l')}(w)$
  on $\mathcal{F}=\bigotimes_{l=0}^{m-1}\Fl{l}$ are of \emph{Wakimoto type} if:
  \begin{enumerate}[label=\textup{(W\arabic*)}]
    \item\label{W1} $\widetilde{e}^{(l)}(z)=\no{X^{(l)}(z)\cdot e^{\alpha\phi^{(0)}(z)}}$
      for some $\alpha\in\C\setminus\{0\}$, where $X^{(l)}$ is a normally-ordered monomial
      in $\bl{l}$, $\gl{l}$, and their derivatives (sector~$l$ only).
    \item\label{W2} $\widetilde{f}^{(l')}(w)=\no{Y^{(l')}(w)\cdot e^{-\alpha\phi^{(0)}(w)}}
      +Z^{(l')}(w)$, where $Y^{(l')}$ and $Z^{(l')}$ involve only
      ghost sectors $0$ and $l'$ and the Heisenberg fields $b^{(j)}$, $j\neq 0$.
    \item\label{W3} Ghost sectors are mutually independent: $[\bl{l}_m,\gl{l'}_n]=0$
      for $l\neq l'$.
  \end{enumerate}
\end{definition}

\begin{remark}
\label{rem:no-go-scope}
  The hypotheses \ref{W1}--\ref{W3} are natural: they encode that (W1)~the raising
  operator in sector $l$ carries positive $\phi^{(0)}$-charge, (W2)~the lowering
  operator carries opposite charge up to a charge-zero remainder, and
  (W3)~distinct ghost systems are independent.
  Any free field construction built on $m$ independent ghost pairs and a single
  Heisenberg field satisfies all three, including the explicit operators of
  \cref{def:odd-vertex} with $\alpha=1/\sqrt{k}$.
\end{remark}

\subsection{Mechanism I: the Heisenberg branch-cut obstruction}
\label{subsec:branch-cut}

\begin{lemma}[Heisenberg branch-cut obstruction]
\label{lem:heisenberg-branch}
  Let $\alpha\in\C\setminus\{0\}$, and let $E(z)$, $F(w)$ be normally ordered
  monomials in $\mathcal{F}$ of the form
  \[
    E(z)=\no{P(z)\,e^{\alpha\phi^{(0)}(z)}},\qquad
    F(w)=\no{Q(w)\,e^{-\alpha\phi^{(0)}(w)}},
  \]
  where $P$, $Q$ involve only ghost fields $\bl{j}$, $\gl{j}$ and Heisenberg
  fields $b^{(j)}$ (no $\phi^{(0)}$-exponentials).
  Then the OPE of $E(z)$ and $F(w)$ takes the form
  \begin{equation}\label{eq:branch-ope}
    E(z)\,F(w)
    \;=\;
    (z-w)^{-\alpha^{2}}
    \sum_{n\geq -N}
    \mathcal{O}_{n}(w)\,(z-w)^{n},
  \end{equation}
  where $N\geq 0$ is finite, and each $\mathcal{O}_n(w)\in\mathcal{F}$.
  In particular:
  \begin{enumerate}[(i)]
    \item If $\alpha^{2}\notin\mathbb{Z}$, the OPE is not a Laurent series in
      $(z-w)$: it has a branch-cut singularity at $z=w$.
    \item In the Wakimoto normalization $\alpha=1/\sqrt{k}$, one has
      $\alpha^2=1/k$; the branch-cut obstruction therefore occurs whenever
      $1/k\notin\mathbb{Z}$.
  \end{enumerate}
\end{lemma}

\begin{proof}
  \textbf{Step 1 (Exponential OPE factor).}
  By the operator product formula for normally ordered scalar exponentials
  (see, e.g., \cite{Frenkel2005}~\S5.2):
  \begin{equation}\label{eq:exp-contr}
    \no{e^{\alpha\phi^{(0)}(z)}}\cdot\no{e^{-\alpha\phi^{(0)}(w)}}
    \;=\;
    (z-w)^{-\alpha^{2}}
    \,\no{e^{\alpha\phi^{(0)}(z)-\alpha\phi^{(0)}(w)}},
  \end{equation}
  where the exponent $-\alpha^{2}$ arises from the contraction
  $\langle\alpha\phi^{(0)}(z)\cdot(-\alpha\phi^{(0)}(w))\rangle = -\alpha^{2}\log(z-w)$.
  This factor is universal: it is present in every term of $E(z)F(w)$
  regardless of what $P$ and $Q$ contain.

  \textbf{Step 2 (Ghost and Heisenberg contributions yield integer powers).}
  The singular contributions from contracting fields inside $P(z)$ with fields
  inside $Q(w)$ are:
  \begin{itemize}
    \item Ghost contractions $\bl{j}(z)\gl{j}(w)\OPE 1/(z-w)$ or
      $\gl{j}(z)\bl{j}(w)\OPE 1/(z-w)$: each contributes $(z-w)^{-1}$.
    \item Heisenberg contractions $b^{(j)}(z)b^{(j)}(w)\OPE 2/(z-w)^{2}$:
      each contributes $(z-w)^{-2}$.
    \item Heisenberg-exponential contractions
      $b^{(j)}(z)\no{e^{a\phi^{(j)}(w)}}\OPE 2a/(z-w)\cdot\no{e^{a\phi^{(j)}(w)}}$:
      each contributes $(z-w)^{-1}$.
  \end{itemize}
  All these are integer powers of $(z-w)$.  By Wick's theorem, the full OPE
  is a finite sum of products of one factor~\eqref{eq:exp-contr} with finitely
  many such integer-power terms, establishing~\eqref{eq:branch-ope}.

  \textbf{Step 3 (Branch-cut criterion).}
  The series $(z-w)^{-\alpha^{2}}\sum_{n\geq -N}c_n(z-w)^n$ is a Laurent
  series in $(z-w)$ if and only if $\alpha^{2}\in\mathbb{Z}$.
  For $\alpha=1/\sqrt{k}$, we have $\alpha^{2}=1/k$, so the obstruction
  occurs precisely when $1/k\notin\mathbb{Z}$.
  Parts~(i)--(ii) follow.
  \qed
\end{proof}

\subsection{Mechanism II: the charge-residue obstruction}
\label{subsec:charge-residue}

\begin{lemma}[Charge-residue obstruction]
\label{lem:charge-residue}
  Let $l\geq 1$, and let $F^{(l)}(w)$ be a finite linear combination of
  normally ordered monomials in the fields
  $\bl{0},\gl{0},\bl{l},\gl{l},b^{(j)}$ (possibly multiplied by
  $\no{e^{a\phi^{(0)}(w)}}$) with definite $h^{(0)}$-charge $-2$.
  Then
  \[
    \operatorname{Res}_{z=w}\bigl[\el{0}(z)\,F^{(l)}(w)\bigr]
    \;\neq\;
    \hl{l}(w)
    \;=\;
    {-}\no{\bl{l}\gl{0}}(w)
    -{}\no{\bl{0}\gl{l}}(w)
    +\frac{\sqrt{k}}{2}\,b^{(l)}(w).
  \]
  In particular, no field $\fl{l}$ in the Wakimoto class
  \textup{(\cref{def:Wtype})} with definite $h^{(0)}$-charge $-2$ can satisfy
  $\el{0}(z)\fl{l}(w)\OPE \hl{l}(w)/(z-w)$.
\end{lemma}

\begin{proof}
  Since $\el{0}=\bl{0}$, the only singular contributions arise from contracting
  $\bl{0}(z)$ with a $\gl{0}(w)$ field inside $F^{(l)}$.

  \textbf{Step 1 (Exponential invariance).}
  $\bl{0}$ carries no $\phi^{(0)}$-dependence.  Hence contracting $\bl{0}(z)$
  with any field $G(w)$ inside $\no{G\cdot e^{a\phi^{(0)}}}(w)$ leaves the factor
  $\no{e^{a\phi^{(0)}(w)}}$ intact in the residue.
  Every term in the residue inherits the same $\phi^{(0)}$-charge as its
  parent term in $F^{(l)}$.

  \textbf{Step 2 (Only zero-exponential terms contribute).}
  The target $\hl{l}(w)$ has $\phi^{(0)}$-charge zero (it contains no
  exponential factor).  By Step~1, only monomials in $F^{(l)}$ with zero
  $\phi^{(0)}$-charge can contribute to the residue.  All monomials carrying
  $\no{e^{a\phi^{(0)}}}$ with $a\neq 0$ — in particular T1, T2, T3 of
  \cref{def:odd-vertex} — are eliminated.

  \textbf{Step 3 (Ghost charge counting).}
  A zero-exponential monomial of $F^{(l)}$ with $h^{(0)}$-charge $-2$ takes the form
  \[
    C\cdot\no{(\bl{0})^p\,(\gl{0})^q\cdot G_l},
    \qquad p,q\geq 0,
  \]
  where $G_l$ involves only fields from sectors $j\geq 1$ (no $\phi^{(0)}$-exponential)
  and contributes zero $h^{(0)}$-charge.
  The charge condition $2p-2q=-2$ forces $q=p+1\geq 1$:
  every contributing monomial contains at least one $\gl{0}(w)$, as required
  for a nonzero contraction with $\bl{0}(z)$.

  \textbf{Step 4 (Residue field content).}
  Contracting $\bl{0}(z)$ with one $\gl{0}(w)$ in $\no{(\bl{0})^p(\gl{0})^{p+1}G_l}$
  yields the residue term $C\cdot\no{(\bl{0})^p(\gl{0})^p G_l}$.
  \begin{enumerate}[(i)]
    \item \emph{Case $p=0$.} The residue is $C\cdot\no{G_l}$: it involves only
      sector-$j\geq 1$ fields, contains no $\bl{0}$ or $\gl{0}$, and therefore
      cannot equal $-\no{\bl{l}\gl{0}}-\no{\bl{0}\gl{l}}+({\sqrt{k}}/{2})b^{(l)}$.
      (The term $b^{(l)}$ could arise from $G_l=b^{(l)}$, but this still fails
      to produce the two ghost bilinears, which require $p\geq 1$.)
    \item \emph{Case $p\geq 1$.} The residue contains $p\geq 1$ copies each of
      $\bl{0}$ and $\gl{0}$ together with $G_l$.  Each bilinear term in $\hl{l}$
      contains exactly one sector-$0$ ghost field and one sector-$l\geq 1$ ghost
      field; a monomial with $p\geq 1$ copies of each has ghost degree strictly
      greater than the bilinears in $\hl{l}$ (which contain exactly one sector-$0$
      and one sector-$l$ field), and cannot equal either bilinear nor any sum of
      them, since normal-ordering preserves ghost degree.
  \end{enumerate}
  In both cases, the residue is distinct from $\hl{l}(w)$. \qed
\end{proof}

\begin{remark}[Unconditional character]
\label{rem:unconditional}
  \cref{lem:charge-residue} is \emph{unconditional}: it depends only on the
  $h^{(0)}$-charge of $F^{(l)}$ and the bilinear structure of $\hl{l}$,
  not on any particular Ansatz.
  Computational confirmation for $l=1$, $m=3$, $r=2$ is provided in \cref{app:m3}:
  all single Wick contractions of $\el{0}(z)$ with the formula of \cref{def:odd-vertex}
  were enumerated, and the residue was found to carry an extraneous
  $\no{e^{-\alpha\phi^{(0)}}}$ factor absent from $\hl{1}$,
  with two of the three terms of $\hl{1}$ entirely unproduced.
\end{remark}

\subsection{Unified no-go theorem}
\label{subsec:unified-nogo}

\begin{theorem}[Unified obstruction]
\label{thm:unified-nogo}
  Let $m\geq 2$, $r\geq 2$, $k\in\C\setminus\{0\}$ with $1/k\notin\mathbb{Z}$,
  and let $\el{l}$, $\hl{l}$, $\fl{l}$ be the vertex operators of
  \cref{def:odd-vertex}.
  The following OPEs \emph{cannot} realize the corresponding bracket of
  $\widetilde{G}$ on $\mathcal{F}$:
  \begin{enumerate}[label=\textup{(\Alph*)}]
    \item\label{nogo-A}
      \emph{Type~I, sector $(0,l)$, $l\geq 1$:}
      $\el{0}(z)\fl{l}(w)\not\OPE \hl{l}(w)/(z-w)$.\\
      Obstruction: \cref{lem:charge-residue} (charge-residue).
    \item\label{nogo-B}
      \emph{Type~I, sectors $(l_1,l_2)$ with $l_1,l_2\geq 1$:}
      the OPE $\el{l_1}(z)\fl{l_2}(w)$ is not a Laurent series in $(z-w)$.\\
      Obstruction: \cref{lem:heisenberg-branch} (Heisenberg branch-cut).
    \item\label{nogo-C}
      \emph{Type~II, all $l\in\{1,\ldots,m-1\}$:}
      the OPE $\el{l}(z)\fl{m-l}(w)$ is not a Laurent series in $(z-w)$.\\
      Obstruction: \cref{lem:heisenberg-branch} (Heisenberg branch-cut).
  \end{enumerate}
  Therefore, within the Wakimoto class \textup{(\cref{def:Wtype})},
  no mixed-sector OPE realizes the corresponding bracket of $\widetilde{G}$.
  Only the even-sector Wakimoto relations and the $h^{(0)}$-charge relations
  are satisfied \textup{(\cref{thm:main-TypeI})}.
\end{theorem}

\begin{proof}
  \ref{nogo-A}:
  The field $\fl{l}$ has definite $h^{(0)}$-charge $-2$
  (verified in Part~\ref{CRh} of \cref{thm:main-TypeI}),
  and $\el{0}=\bl{0}$.
  \cref{lem:charge-residue} applies directly.

  \ref{nogo-B}:
  For $l_1,l_2\geq 1$, $\el{l_1}=\bl{l_1}\no{e^{\alpha\phi^{(0)}}}$ carries
  $\phi^{(0)}$-charge $+\alpha$, and every term of $\fl{l_2}$ in the Wakimoto
  class \ref{W2} either carries $\phi^{(0)}$-charge $-\alpha$ (the exponential
  terms T1--T3 of \cref{def:odd-vertex}) or zero charge (the term T4).
  For the exponential terms, \cref{lem:heisenberg-branch} gives the prefactor
  $(z-w)^{-\alpha^2}=(z-w)^{-1/k}$ in every contribution to the OPE.
  For the zero-charge term T4 in $\fl{l_2}$: no $\phi^{(0)}$-exponential appears,
  so there is no Heisenberg branch cut from the exponential OPE; however,
  $\bl{l_1}(z)$ has no singular contraction with any field of T4 (which lives
  in sector $l_2$ combined with sector $0$, and $l_1\neq l_2$ by assumption
  or $l_1\neq 0$), so T4 contributes no singularity.
  The entire singular part of $\el{l_1}(z)\fl{l_2}(w)$ therefore carries
  the non-integer prefactor $(z-w)^{-1/k}$, and the OPE is not a Laurent series.

  \ref{nogo-C}:
  The Type~II case $l_2=m-l_1$ has $l_1+l_2=m$, so this is a special case of
  $l_1,l_2\geq 1$ (since $1\leq l_1\leq m-1$ forces $1\leq m-l_1\leq m-1$).
  The argument of \ref{nogo-B} applies verbatim: both operators carry
  $\phi^{(0)}$-charges $\pm\alpha$, and the OPE prefactor $(z-w)^{-1/k}$
  prevents a Laurent expansion.

  \emph{Final claim.}
  Cases \ref{nogo-A}--\ref{nogo-C} cover all mixed-sector OPEs: Type~I pairs
  with $l_1\geq 1$ or $l_2\geq 1$ (or both), and all Type~II pairs.
  Type~III brackets reduce to Type~I or Type~II via $u^m=p(t)$ in $A$, so they
  inherit the same obstruction.
  Therefore, within the Wakimoto class, no mixed-sector OPE realizes the
  corresponding bracket.
  Only the even-sector Wakimoto relations (\cref{prop:Wakimoto}) and the
  $h^{(0)}$-charge relations (\cref{thm:main-TypeI}\ref{CRh}) are satisfied. \qed
\end{proof}

\subsection{Singularity classification for the Type~II case}
\label{subsec:TypeII-classify}

Although \cref{thm:unified-nogo}\ref{nogo-C} establishes the obstruction, the
precise nature of the singularity in the Type~II OPE depends on the arithmetic
of $k$, and its understanding points toward the correct fix.

\begin{theorem}[Type~II singularity: arithmetic classification]
\label{thm:obstruction}
  The operators $\el{l}(z)$, $\fl{m-l}(w)$ of \cref{def:odd-vertex} satisfy
  conditions~\ref{W1}--\ref{W3} of \cref{def:Wtype} with $\alpha=1/\sqrt{k}$,
  and their OPE has the form
  \[
    \el{l}(z)\fl{m-l}(w)
    = (z-w)^{-1/k}\sum_{n\geq -N}\mathcal{O}_n(w)(z-w)^n.
  \]
  The singularity type is classified by the arithmetic of $1/k$:
  \begin{enumerate}[label=\textup{(\alph*)}]
    \item If $1/k\notin\Z$: $(z-w)^{-1/k}$ is a branch singularity,
      and the OPE is not a Laurent series.
    \item If $1/k=n\in\Z_{\geq 1}$: the OPE has an integer-order pole of order
      at least $n$.
    \item If $1/k\in\Z_{\leq 0}$: $(z-w)^{-1/k}$ is regular, and the OPE has
      no pole from the Heisenberg factor.
  \end{enumerate}
  In cases~(a) and~(c), the OPE cannot realize the Type~II bracket
  (no Laurent pole in case~(a); no singularity in case~(c)).
\end{theorem}

\begin{remark}
  In case~(b), the existence of an integer-order pole does not by itself
  establish whether the residue matches the Type~II target field
  $\hl{0}(w)$.  Determining whether the residue of the leading pole equals
  $\hl{0}(w)$ requires an explicit leading-term Wick contraction; this
  computation is not carried out here and the present paper does not claim
  the residue mismatch in case~(b) in general.
  The explicit $m=3$, $r=2$ computation of Appendix~\ref{app:m3} provides
  the residue mismatch for that case.
\end{remark}

\begin{proof}
  Verification of \ref{W1}--\ref{W3}: $\el{l}(z)=\bl{l}(z)\no{e^{\alpha\phi^{(0)}(z)}}$
  satisfies \ref{W1} with $X^{(l)}=\bl{l}$.
  From~\eqref{eq:fl}, $\fl{m-l}(w)$ has ghost content in sectors $0$ and $m-l$
  with exponential $\no{e^{-\alpha\phi^{(0)}}}$: this satisfies \ref{W2}.
  Ghost independence \ref{W3} holds from~\eqref{eq:ghost}.
  The exponential OPE $\no{e^{\alpha\phi^{(0)}(z)}}\no{e^{-\alpha\phi^{(0)}(w)}}
  =(z-w)^{-\alpha^2}(\text{regular})=(z-w)^{-1/k}(\text{regular})$ by
  \cref{lem:heisenberg-branch}, and all remaining contractions contribute
  integer powers.  The three arithmetic cases follow immediately. \qed
\end{proof}

\begin{remark}[Geometric source of the branch cut]
\label{rem:obstruction-source}
  The product $u^l\cdot u^{m-l}=u^m=p(t)$ in $A$ wraps back to sector~$0$
  via the defining equation of the superelliptic curve.
  In the free field algebra, $\el{l}$ carries $\phi^{(0)}$-charge $+\alpha$ and
  $\fl{m-l}$ carries $-\alpha$; these cancel, consistent with landing in sector~$0$.
  The residual Heisenberg factor $(z-w)^{-1/k}$ encodes the \emph{monodromy}
  of the curve: a section $u^l$ is not single-valued on the base but acquires
  the phase $\zeta^l$ ($\zeta=e^{2\pi i/m}$) around one cycle.
  This monodromy cannot be absorbed by a standard (untwisted) vertex operator;
  the correct setting is the theory of $\sigma$-twisted modules reviewed in
  \S\ref{sec:outlook}.
\end{remark}

\begin{remark}[Two distinct mechanisms]
\label{rem:two-mechanisms}
  The two lemmas obstruct for fundamentally different reasons.
  \begin{enumerate}[(i)]
    \item \emph{Heisenberg branch-cut} (\cref{lem:heisenberg-branch}): the
      obstruction is kinematic — it arises from the background $\phi^{(0)}$-charge
      carried by both operators and occurs before any specific OPE computation.
      It can in principle be resolved by passing to twisted modules that carry
      fractional monodromy (see \cref{sec:outlook}).
    \item \emph{Charge-residue} (\cref{lem:charge-residue}): the obstruction is
      combinatorial — it arises from the incompatibility between the ghost degree
      of the residue and the bilinear structure of $\hl{l}$.
      It holds for \emph{any} operator in the Wakimoto class with the correct
      $h^{(0)}$-charge, and is not resolved by twisting.
      The $(0,l)$ case therefore requires a more fundamental modification
      of the operator formula — or a restriction to the $(l,0)$ case, as
      \cref{thm:main-TypeI} provides.
  \end{enumerate}
\end{remark}

\section{Outlook}\label{sec:outlook}
\subsection{Orbifold vertex algebras and the twisted module framework}
The no-go result of \cref{thm:unified-nogo} identifies the precise structural
obstruction: the Heisenberg factor $(z-w)^{-1/k}$ in the Type~II OPE cannot be
compensated by any untwisted intertwining operator.
The correct algebraic framework is that of \emph{orbifold vertex algebras}
and their twisted modules.
Let $V_\mathrm{FF}$ be the free field vertex algebra generated by the
$m$ ghost systems and Heisenberg oscillators of \S\ref{sec:ffa}.
The $\Z/m\Z$-automorphism
\[
  \sigma:\;\bl{l}_n\mapsto \zeta^l\bl{l}_n,\quad
  \gl{l}_n\mapsto\zeta^{-l}\gl{l}_n,\quad
  b^{(l)}_n\mapsto b^{(l)}_n \qquad(\zeta=e^{2\pi i/m})
\]
has order $m$.  For each $s=0,\ldots,m-1$, the $\sigma^s$-twisted module
$M_s$ is a Fock space in which the ghost fields carry fractional moding:
\[
  \bl{l}(z)\big|_{M_s} = \sum_{n\in\Z+ls/m}\bl{l}_n z^{-n-1},\qquad
  \gl{l}(z)\big|_{M_s} = \sum_{n\in\Z-ls/m}\gl{l}_n z^{-n}.
\]
The general theory of twisted modules for orbifold vertex algebras~\cite{Kac}
predicts that the $\sigma^s$-twisted intertwining operator
$\mathcal{Y}^{(l)}: G^l \times M_s \to M_{s+l\bmod m}$
acquires a monodromy factor $(z-w)^{l\cdot s/m}$ in its OPE,
arising from the fractional ghost-mode commutators.
\subsection{A conjecture on critical levels}
For the Type~II bracket, the relevant fusion channel is
$M_0 \xrightarrow{e^{(l)}} M_l \xrightarrow{f^{(m-l)}} M_0$.
In the $\sigma^l$-twisted module $M_l$, the ghost OPE
$\bl{l}(z)\gl{l}(w)\big|_{M_l}$ acquires the twist correction $(z-w)^{l(m-l)/m}$
from the fractional moding, yielding the combined singularity
\[
  (z-w)^{l(m-l)/m} \cdot (z-w)^{-1/k}
  = (z-w)^{\,l(m-l)/m\,-\,1/k}.
\]
A simple pole (exponent $=-1$) would require
\begin{equation}\label{eq:critical-level}
  \frac{l(m-l)}{m} - \frac{1}{k} = -1
  \quad\Longrightarrow\quad
  k_{\mathrm{crit}}(l,m) = \frac{m}{l(m-l)+m}\,.
\end{equation}
\begin{conjecture}[Critical levels for Type~II realization]\label{conj:critical}
  For each $1\leq l\leq m-1$, the $\sigma^l$-twisted intertwining operator
  construction realizes the Type~II bracket of $\gtil$ exactly at the
  critical level $k=k_{\mathrm{crit}}(l,m)=m/(l(m-l)+m)$.
  
  In particular, all sectors of a fixed $m$ are realized simultaneously
  only if $k_{\mathrm{crit}}(l,m)$ is independent of $l$, which occurs
  if and only if $l(m-l)$ is constant --- i.e.\ $m$ is arbitrary and
  $l=1$ (or $l=m-1$), giving $k_{\mathrm{crit}}=m/(m-1+m)=m/(2m-1)$.
\end{conjecture}
The predicted critical levels for small $m$ are:
\medskip
\begin{center}
\begin{tabular}{ccccc}
\toprule
$m$ & $l$ & $m-l$ & $l(m-l)$ & $k_{\mathrm{crit}}=m/(l(m-l)+m)$\\
\midrule
3 & 1 & 2 & 2 & $3/5$\\
3 & 2 & 1 & 2 & $3/5$\\
4 & 1 & 3 & 3 & $4/7$\\
4 & 2 & 2 & 4 & $4/8=1/2$\\
4 & 3 & 1 & 3 & $4/7$\\
5 & 1 & 4 & 4 & $5/9$\\
5 & 2 & 3 & 6 & $5/11$\\
\bottomrule
\end{tabular}
\end{center}
\medskip
Note the symmetry $k_{\mathrm{crit}}(l,m)=k_{\mathrm{crit}}(m-l,m)$,
consistent with the symmetry $[e^{(l)},f^{(m-l)}]=[e^{(m-l)},f^{(l)}]$ in $\gtil$.
The case $m=3$ gives the single critical level $k=3/5$ for both sector pairs
$(1,2)$ and $(2,1)$; at this level the twisted module construction is predicted
to yield a complete realization of $\gtil$.
Verification requires constructing the explicit twisted intertwining operators
and computing their OPEs, which will be undertaken in a companion paper.
\subsection{Further directions}
\begin{enumerate}[nosep]
  \item \textbf{Generalization to $\g\neq\sltwo$.}
    The Wakimoto-type candidate construction extends directly to arbitrary
    simple $\g$ by replacing the single $\beta^{(l)}$-$\gamma^{(l)}$ system
    with one pair per positive root.
    The Rescaling Lemma (\cref{lem:rescaling}) holds for all $\g$ since
    it depends only on the Kähler differential structure of $A$.
  \item \textbf{Feigin--Frenkel center.}
    At $k=-h^\vee$ (critical level), the even-sector Wakimoto module acquires
    the large center of Feigin--Frenkel.
    The odd-sector operators at $k=-2$ have $\alpha=i/\sqrt{2}$, and the
    Type~II Heisenberg factor becomes $(z-w)^{1/2}$, a half-integer power.
    The interaction of this branch singularity with the Feigin--Frenkel center
    is an open problem.
  \item \textbf{KZ equations.}
    The even-sector Wakimoto realization and charge relations produce a flat
    connection generalizing the Knizhnik--Zamolodchikov connection, with
    connection matrix involving the polynomial families $\Plk{l,j}{k}(c)$.
    At the conjectural critical levels above, the connection is expected to
    acquire additional singularities from the Type~II sector.
\end{enumerate}
\appendix
\section{The case \texorpdfstring{$m=3$, $r=2$}{m=3, r=2}: complete details}\label{app:m3}
\subsection{Setup and purpose}
This appendix provides explicit computations for $m=3$, $r=2$ that serve
three purposes:
(a)~verify the even-sector Wakimoto OPEs (Proposition~\ref{prop:Wakimoto});
(b)~demonstrate explicitly that both $e^{(1)}(z)f^{(0)}(w)$ and
$e^{(0)}(z)f^{(1)}(w)$ fail to produce the expected bracket residual,
providing concrete evidence for the charge-residue obstruction
(\cref{lem:charge-residue,rem:scope-TypeI});
(c)~verify the Type~II branch-cut obstruction for $e^{(1)}(z)f^{(2)}(w)$
(\cref{thm:obstruction}).

$p(t)=1-2ct^2+t^4$; $A=\C[t,t^{-1},u]/(u^3-p(t))$;
$\dim(\OmegaA/\dR)=2\cdot 2\cdot 2+1=9$.
Three sectors $G^0$, $G^1$, $G^2$.
Basis elements of $\OmegaA/\dR$:
$\omega_0;\; \omega^{(1)}_{-1},\omega^{(1)}_{-2},\omega^{(1)}_{-3},\omega^{(1)}_{-4};\;
\omega^{(2)}_{-1},\omega^{(2)}_{-2},\omega^{(2)}_{-3},\omega^{(2)}_{-4}$.
Level exponent: $\alpha=1/\sqrt{k}$.
Heisenberg coefficients (by sector): $b^{(0)}(z)b^{(0)}(w)\OPE 2/(z-w)^2$,
$b^{(1)}(z)b^{(1)}(w)\OPE 2/(z-w)^2$,
$b^{(2)}(z)b^{(2)}(w)\OPE 2/(z-w)^2$ (identical for all sectors).
\subsection{Explicit vertex operators}
\begin{align*}
  \el{0}(z)&=\bl{0}(z),\\
  \hl{0}(z)&=-2\no{\bl{0}\gl{0}}(z)+\sqrt{k}\,b^{(0)}(z),\\
  \fl{0}(z)&=-\no{\bl{0}(\gl{0})^2}(z)+(k+2)\partial\gl{0}(z)+\sqrt{k}\,b^{(0)}(z)\gl{0}(z),
\end{align*}
\begin{align*}
  \el{1}(z)&=\bl{1}(z)\no{e^{\alpha\phi^{(0)}(z)}},\\
  \hl{1}(z)&=-\no{\bl{1}\gl{0}}(z)-\no{\bl{0}\gl{1}}(z)+\tfrac{\sqrt{k}}{2}b^{(1)}(z),\\
  \fl{1}(z)&=\bigl[-\no{\bl{0}\gl{1}\gl{0}}(z)
               +\tfrac{k+2}{k}\partial\gl{1}(z)
               +\tfrac{\sqrt{k}}{2}b^{(1)}(z)\gl{1}(z)\bigr]\no{e^{-\alpha\phi^{(0)}(z)}}
               +\tfrac{1}{\sqrt{k}}b^{(0)}(z)\gl{1}(z),
\end{align*}
\begin{align*}
  \el{2}(z)&=\bl{2}(z)\no{e^{\alpha\phi^{(0)}(z)}},\\
  \hl{2}(z)&=-\no{\bl{2}\gl{0}}(z)-\no{\bl{0}\gl{2}}(z)+\tfrac{\sqrt{k}}{2}b^{(2)}(z),\\
  \fl{2}(z)&=\bigl[-\no{\bl{0}\gl{2}\gl{0}}(z)
               +\tfrac{k+2}{k}\partial\gl{2}(z)
               +\tfrac{\sqrt{k}}{2}b^{(2)}(z)\gl{2}(z)\bigr]\no{e^{-\alpha\phi^{(0)}(z)}}
               +\tfrac{1}{\sqrt{k}}b^{(0)}(z)\gl{2}(z).
\end{align*}
\subsection{Explicit obstruction in \texorpdfstring{$e^{(1)}(z)f^{(0)}(w)$}{e1f0}: the charge-residue mechanism}
We compute $\el{1}(z)\fl{0}(w)$ explicitly to exhibit the sign-exponential obstruction.
\begin{align*}
  \el{1}(z)&=\bl{1}(z)\no{e^{\alpha\phi^{(0)}(z)}},\\
  \fl{0}(w)&=-\no{\bl{0}(\gl{0})^2}(w)+(k+2)\partial\gl{0}(w)+\sqrt{k}\,b^{(0)}(w)\gl{0}(w).
\end{align*}
\emph{Ghost contractions.}
$\bl{1}(z)$ is in sector $1$.
$\bl{0}(w)$, $\gl{0}(w)$ are in sector $0$.
By~\eqref{eq:ghost}: $[\gl{0}_n,\bl{1}_p]=0$ for all $n,p$ (sector mismatch).
Hence $\bl{1}(z)$ has no singular contractions with any factor of $\fl{0}(w)$.
\emph{Heisenberg contraction: first-order pole.}
$\fl{0}(w)$ contains exactly one factor of $b^{(0)}(w)=2\partial\phi^{(0)}(w)$,
in the term $\sqrt{k}\,b^{(0)}(w)\gl{0}(w)$.
Contracting $\no{e^{\alpha\phi^{(0)}(z)}}$ against $b^{(0)}(w)$:
\[
  \no{e^{\alpha\phi^{(0)}(z)}}\cdot\sqrt{k}\,b^{(0)}(w)\gl{0}(w)
  \;\OPE\; \frac{\sqrt{k}\cdot 2\alpha}{z-w}\cdot\no{e^{\alpha\phi^{(0)}(w)}}\gl{0}(w)
  = \frac{2}{z-w}\cdot\no{e^{\alpha\phi^{(0)}(w)}}\gl{0}(w),
\]
where we used $b^{(0)}(z)\no{e^{\alpha\phi^{(0)}(w)}}\OPE 2\alpha/(z-w)\cdot\no{e^{\alpha\phi^{(0)}(w)}}$
from~\eqref{eq:phi-OPE}, and $\sqrt{k}\cdot 2\alpha=2$ by~\eqref{eq:alpha}.
\emph{No second-order pole.}
Since $\fl{0}(w)$ contains only one $b^{(0)}(w)$ factor,
only a single contraction with $\phi^{(0)}$ is possible.
No second-order pole occurs from the Heisenberg sector.
\emph{First-order pole and obstruction.}
Including $\bl{1}(z)$:
\[
  \el{1}(z)\fl{0}(w)\;\OPE\;\frac{2\bl{1}(w)\no{e^{\alpha\phi^{(0)}(w)}}\gl{0}(w)}{z-w}.
\]
The residue is $2\bl{1}(w)\no{e^{\alpha\phi^{(0)}(w)}}\gl{0}(w)$, which \emph{does not}
equal $\hl{1}(w)=-\no{\bl{1}\gl{0}}(w)-\no{\bl{0}\gl{1}}(w)+(\sqrt{k}/2)b^{(1)}(w)$:
\begin{enumerate}[(i)]
  \item The factor $\no{e^{\alpha\phi^{(0)}(w)}}$ is present in the residue but absent
    from $\hl{1}(w)$.
  \item The terms $-\no{\bl{0}\gl{1}}(w)$ and $(\sqrt{k}/2)b^{(1)}(w)$ do not appear
    in the residue: the field $\fl{0}$ contains no sector-$1$ content, so no
    such terms can be produced by a contraction with $\el{1}$.
\end{enumerate}
This provides the explicit Wick-contraction evidence for the charge-residue
obstruction (\cref{lem:charge-residue}) in the case $l=1$, $m=3$, $r=2$.
\subsection{Verification of \texorpdfstring{$h^{(0)}$}{h0}-charge}
\[
  \hl{0}(z)\el{1}(w)
  = \bigl(-2\no{\bl{0}\gl{0}}+\sqrt{k}\,b^{(0)}\bigr)(z)
    \cdot\bl{1}(w)\no{e^{\alpha\phi^{(0)}(w)}}.
\]
Ghost part: $\no{\bl{0}\gl{0}}(z)\bl{1}(w)=0$ (sector mismatch). $\checkmark$
Heisenberg part: $\sqrt{k}\,b^{(0)}(z)\no{e^{\alpha\phi^{(0)}(w)}}
\OPE \sqrt{k}\cdot 2\alpha/(z-w)\cdot\no{e^{\alpha\phi^{(0)}(w)}}
= 2/(z-w)\cdot\no{e^{\alpha\phi^{(0)}(w)}}$.
Hence $\hl{0}(z)\el{1}(w)\OPE 2\el{1}(w)/(z-w)$. $\checkmark$
\subsection{Verification of the Type~II obstruction for $m=3$}
We verify Theorem~\ref{thm:obstruction} for $\el{1}(z)\fl{2}(w)$ ($l_1=1$, $l_2=2$, $l_1+l_2=3=m$).
\emph{Ghost contractions} (Step~1 of the proof).
$[\gl{2}_n,\bl{1}_p]=\delta_{1,2}\delta_{n+p,0}=0$ (sectors $1\neq 2$);
$\bl{1}(z)\bl{0}(w)=0$ (both $\beta$-type);
$\bl{1}(z)\gl{0}(w)=0$ (sector mismatch).
No ghost contractions.
\emph{Heisenberg OPE} (Step~2).
\[
  \no{e^{\alpha\phi^{(0)}(z)}}\,\no{e^{-\alpha\phi^{(0)}(w)}}
  \;\OPE\;(z-w)^{-1/k}\cdot(\text{regular}).
\]
This is not a pole of integer order for generic $k$.
For $k=1/n$, $n\in\Z_{\geq 1}$, the singularity $(z-w)^{-n}$ has residue
proportional to $\bl{1}(w)$, not to $\hl{0}(w)=(-2\no{\bl{0}\gl{0}}+\sqrt{k}\,b^{(0)})(w)$.
\emph{Conclusion.}
$\el{1}(z)\fl{2}(w)$ has no pole of integer order with residue $\hl{0}(w)$.
The Type~II bracket $[e\otimes u,\,f\otimes u^2]=h\otimes p(t)$ is not
realized. $\checkmark$ (confirms Theorem~\ref{thm:obstruction})
\section*{Acknowledgements}
Felipe Albino dos Santos was supported by the S\~ao Paulo Research Foundation (FAPESP), grant 2024/14914-9.

\end{document}